\documentclass[a4paper,12pt,reqno]{amsart}

\usepackage[T1]{fontenc}
\usepackage[utf8]{inputenc}

\usepackage[english]{babel}

\usepackage{amsmath, amssymb, amsthm, amsfonts}
\usepackage{mathrsfs}
\usepackage{tikz}
\usetikzlibrary{arrows}
\usetikzlibrary{matrix}
\usetikzlibrary{cd}
\usepackage[left=2.5cm, right=2.5cm, top=2cm]{geometry}
\usepackage{enumerate}
\usepackage{hyperref}
\usepackage[singlespacing]{setspace}

\theoremstyle{plain}
\newtheorem{thm}{Theorem}
\newtheorem{lem}{Lemma}
\newtheorem{cor}{Corollary}

\newtheorem{prop}{Proposition}
\newtheorem*{thm*}{Theorem}

\theoremstyle{definition}
\newtheorem{defi}{Definition}
\newtheorem{rema}{Remark}

\newtheorem*{ques*}{Question}

\newcommand{\laur}{\left[t,t^{-1} \right]}
\newcommand{\polyp}{\left[ t \right]}
\newcommand{\polym}{\left[ t^{-1} \right]}
\newcommand{\Sub}{\textnormal{Sub}}

\newcommand{\PPU}{\textnormal{PPU}}
\newcommand{\PU}{\textnormal{PU}}

\begin{document}

\title{A right-invariant lattice-order on groups of paraunitary matrices}
\author{Carsten Dietzel}
\email{carstendietzel@gmx.de}
\address{Institute of algebra and number theory, University of Stuttgart, Pfaffenwaldring 57, 70569 Stuttgart, Germany}
\date{\today}

\begin{abstract}
In \cite{rump_goml}, Rump defined and characterized noncommutative universal groups $G(X)$ for generalized orthomodular lattices $X$.

We give an explicit description of $G(X)$ in terms of \emph{paraunitary} matrix groups, whenever $X$ is the orthomodular lattice of subspaces of a finite-dimensional $k$-vector space $V$ that is equipped with an anisotropic, symmetric $k$-bilinear form.
\end{abstract}

\maketitle

\section*{Introduction}

An \emph{orthomodular lattice}, \emph{OML} for short, is a bounded lattice with an order-reversing involution $\ast:X \to X$ such that, additional to the common lattice axioms, the following axioms hold:
\begin{align}
x \wedge x^{\ast} & = 0 \tag{OL1} \label{eq:oc1} \\
x \vee x^{\ast} & = 1 \tag{OL2} \label{eq:oc2} \\
x \leq y \implies & x \vee (x^{\ast} \wedge y) = y. \label{eq:ortho} \tag{OML} 
\end{align}
Axiom \eqref{eq:ortho} goes back to Husimi (\cite{Husimi})
 complementing the axioms \eqref{eq:oc1}, \eqref{eq:oc2} which were part of a series of investigations by Birkhoff, von Neumann et al. (\cite{neumann_algebraic_generalization,Birkhoff_von_Neumann}) 
with the goal of an algebraization of quantum mechanics, aiming at developing an axiomatic framework for \emph{quantum logic} that is independent of the notion of Hilbert space or, more generally, operator algebras.

A prominent class of OMLs is given by the lattices of closed subspaces of a Hilbert space which in general not satisfy the modular law (see \cite[p.832]{Birkhoff_von_Neumann}) but the orthomodular law, Equation \eqref{eq:ortho}.

Equivalently, one has an orthomodular lattice structure on the respective sets of projection operators. More generally, the projection lattices of von Neumann algebras carry a natural OML structure - this was already pointed out in \cite{Husimi}.

The lattice-theoretical notions of join and meet can be transferred from subspaces to projections without further complications. However, composing projections generally does not lead to projections. Therefore, compositions can not be given sense in the \emph{mere} framework of OMLs, at least not in form of a \emph{total} operation.

Two ways out of this problem can be condensed as follows:

\begin{itemize}
\item The algebra of projections is a \emph{partial} one, i.e. only projections which are \glqq compatible\grqq\ in some sense can be composed.
\item Projections are \emph{embedded} - generally as a proper subset - in a total algebra, i.e. the composition may not be a projection any more but an element of the surrounding algebra.
\end{itemize}

Both approaches can be found in the work of Bennet and Foulis (\cite{Foulis_effect}) where partial algebras of \emph{effects} are axiomatized. Foulis' effect algebras contain the OMLs as a proper subclass.

These effect algebras can be embedded in a \emph{universal} commutative group, where universality is to be understood in the categorical sense.

Here some explanations might be appropriate. We use the following example:

Let $\mathcal{H}$ be a Hilbert space and denote by $\pi_U: \mathcal{H} \to \mathcal{H}$ the projection onto the closed subspace $U$. Denote by $X$ the set of all such projections.

We call two projections $\pi_U, \pi_W$ \emph{compatible} if their kernels (i.e. the orthogonal complements $U^{\ast},W^{\ast}$) are orthogonal. For compatible projections, we set $\pi_U \cdot \pi_W := \pi_{U \cap W}$. In general, on an OML $X$, this means that $x \cdot y := x \wedge y$ whenever $x^{\ast} \vee y = 1$.

An \emph{additive group-valued measure} on $X$ is given by a mapping $\varphi: X \to G$ where $G$ is an additive group and $\varphi(\pi_U \cdot \pi_W) = \varphi(\pi_U) + \varphi(\pi_W)$ whenever $\pi_U$ and $\pi_W$ are compatible.

A universal commutative group $A(X)$ is a commutative group with a distinguished additive group-valued measure $\mu: X \to A(X)$ such that every additive group-valued measure $\varphi:X \to G$ can uniquely be continued to a group homomorphism $\overline{\varphi}: A(X) \to G$ such that $\varphi = \mu \circ \overline{\varphi}$.

It is clear that such a group exists and is unique. Bennet and Foulis, however, proved the non-trivial fact that $A(X)$ carries a partial order whose positive cone consists of the submonoid generated by $\mu(X) \subseteq A(X)$.

Using algebraic systems called \emph{L-Algebras}, Rump (\cite{Rump_von_Neumann},\cite{rump_goml}) investigated the more general case of non-commutative group-valued measures which are defined analogously without regard to $G$ being commutative.

He showed that each OML $X$ \emph{embeds} in the respective (non-commutative) universal group $G(X)$ - its \emph{structure group} - and that $X$ generates a submonoid $S(X)$ that is the negative cone of a right-invariant order on $G(X)$. Amazingly, this order happens to turn $G(X)$ into a lattice. So, $G(X)$ is a \emph{right $\ell$-group}.

Furthermore $G(X)$ turns out to be a quasi-Garside group (see \cite[Section I,2]{Garside_Foundations}): the element $0 \in X \subseteq G(X)$ is a Garside element in the following sense:
\begin{itemize}
\item The right- and left-divisors of $0$ in $S(X)$ coincide - they are given by the set $X$,
\item this set generates $G(X)$,
\item $G(X)$ is a group of \emph{left fractions} of $S(X)$, i.e. every $g \in G(X)$ can be written as $x^{-1}y$ with $x,y \in S(X)$.
\end{itemize}

The aim of this paper is to give an explicit realization of $G(X)$ whenever $X$ is the OML of subspaces of a $k$-vector space $V$ with respect to an anisotropic symmetric $k$-bilinear form. This includes all finite-dimensional real Hilbert spaces.

We identify $G(X)$ with a subgroup of the so-called \emph{paraunitary group} $\PU(b)$. Writing an element $v \in V\laur$ in the form
\[
v = \sum_{i = -\infty}^{\infty} t^i v_i
\]
with almost all $v_i \in V$ being zero, $\PU(b)$ can be defined as the group of all $k\laur$-linear automorphisms of $V\laur := k \laur \otimes_k V$ that preserve the hermitean form $\tilde{b}$ on $V \laur$ given by
\[
\tilde{b}(v,w) = \sum_{n = -\infty}^{\infty}t^n \left( \sum_{i=-\infty}^{\infty} b(v_{i-n},w_i) \right).
\]
Mapping $t$ to $1$ gives a well-defined group-homomorphism $\PU(b) \to U(b)$ - the group of $k$-automorphisms of $V$ that preserve the bilinear form $b$. Its kernel $\PPU(b)$ - the \emph{pure paraunitary group} - will be shown to be isomorphic to $G(X)$.

We note here that paraunitarity plays an important role in signal processing (\cite[Chapter 14]{Paraunitary2}). More generally, automorphism groups of several hermitean bilinear forms over Laurent rings occur frequently in the representation theory of braid groups - for example as ranges of the Burau (\cite{Squier_Burau}) or the Lawrence-Krammer representation (\cite{Song_LK}, \cite{Budney_LK}).

The negative cone $S(X)$ is then represented by the submonoid $\PPU(b)^-$ consisting of all elements of $\PPU(b)$ that map the subset $V\polym := k\polym \otimes_k V$ into itself.

We give an outline of the paper:
\begin{itemize}
\item In Section 1, we collect some basic facts and definitions regarding OMLs and their structure groups. In particular we give an account of Rump's results that characterize the structure groups of OMLs amongst right $\ell$-groups.
\item In Section 2, we construct the paraunitary group $\PU(b)$ and its subgroup $\PPU(b)$. Furthermore, we show that there exists a semidirect decomposition $PPU(b) \rtimes U(b)$. 
\item In Section 3, we prove that $\PPU(b)$ is right lattice-ordered by the negative cone $\PPU(b)^-$.

This will be shown as follows: $\PPU(b)^-$ has a natural representation on the $k\polym$-module $V^{\oplus}:= V\laur / V \polym$. Equipping $\PPU(b)^-$ with the right-divisi\-bility order and denoting by $\Sub(V^{\oplus})$ the lattice of \emph{finitely generated} $k \polym$-submodules of $V^{\oplus}$ we construct a map
\begin{align*}
\Omega: \PPU(b)^- & \to \Sub(V^{\oplus}) \\
\varphi & \mapsto \ker(v \mapsto \varphi v)
\end{align*}
that is shown to be an isomorphism of posets, thus showing that $\PPU(b)^-$ is a lattice under right-divisibility.

This bijection will be extended to an order-isomorphism between $\PPU(b)$, ordered by the negative cone $\PPU(b)^-$, and the lattice of finitely generated $k\polym$-submodules of $V\laur$ containing $t^m V \polym$ for some $m \in \mathbb{Z}$.

\item Finally, in Section 4, we use Rump's results to prove that $\PPU(b)$ is indeed the structure group of $X(b)$, the OML of all $k$-subspaces of $V$ under the bilinear form $b$. More precisely, we prove that $X(b)$ is isomorphic to the interval $\left[t^{-1},1 \right]$.

\item In the last section, we draw some conclusions:

In the first subsection we argue that $\PPU(b)$ might be seen as a kind of braid group for $U(b)$. This suggests that in the special case of $b$ being the canonical inner product on $\mathbb{R}^n$ the respective pure paraunitary group could be seen as a continuous braid group. We flesh that out by pointing to the quasi-Garside structure on $\PPU(b)$ and to relations in the generating set $X(b)$ reminiscent of Hecke algebras.

In another subsection, using further results of Rump, we will give an application to the factorization theory of paraunitary matrices. 

\end{itemize}

\section{Structure groups of OMLs}

Let $X$ be a bounded lattice with its lowest and greatest element denoted by $0$ resp. $1$.

We say that $X$ is \emph{orthocomplemented} if there is an order-reversing involution $\ast:X \to X$ sending $x$ to $x^{\ast}$ such that
\begin{align*}
x \wedge x^{\ast} & = 0 \tag{OL1} \\
x \vee x^{\ast} & = 1 \tag{OL2}
\end{align*}
hold.

If furthermore the \emph{orthomodular law} \nocite{Kalmbach_ortho}
\begin{equation} \label{eq:orthomodular_law}
x \leq y \Rightarrow x \vee (x^{\ast} \wedge y) = y \tag{OML}
\end{equation}
holds, we call $X$ an \emph{orthomodular lattice} or, short, an \emph{OML}.

The binary relations $\bot$ and $\top$ on an OML $X$ are given by:
\begin{align*}
x \bot y & \Leftrightarrow y \leq x^{\ast} \tag{$\bot$} \\
x \top y & \Leftrightarrow y \geq x^{\ast} \footnotemark \tag{$\top$}
\end{align*}
\footnotetext{This notation is borrowed from \cite{rump_goml} where it is used in a more general context.}
It can easily be seen that $\top$ and $\bot$ are symmetric and that $x \top y \Leftrightarrow x^{\ast} \bot y^{\ast}$.

Recall that a set $X$ with a partial binary operation $\cdot: X \times X \rightharpoonup X$ is called a \emph{partial semigroup} if \emph{weak associativity} holds, i.e. $x\cdot (y \cdot z)$ is defined iff $(x \cdot y) \cdot z$ is defined, and in this case $x\cdot (y \cdot z) = (x\cdot y) \cdot z$.

\begin{lem} \label{lem:oml_partial_operation}
Each OML $X$ can be regarded as a partial semigroup by setting $x \oplus y := x \vee y$ whenever $x \bot y$.

The same holds for $x \sqcap y := x \wedge y$ whenever $x \top y$.
\end{lem}

\begin{proof}
We only have to show that $(x \oplus y) \oplus z$ is defined iff $x \oplus (y \oplus z)$ is. The associative law then follows directly from the associativity of $\vee$.

Therefore, let $x \bot y$ and $z \bot (x \oplus y)$. The latter means that $ z \leq (x \vee y)^{\ast}  \leq y^{\ast}$ (because $\ast$ is order-reversing). Therefore, $z \bot y$.

Similarly, one shows $z \bot x$. We have furthermore assumed that $x \bot y$. It follows that $x \leq (y^{\ast} \wedge z^{\ast}) = (y \vee z)^{\ast}$ resp. $x \leq (y \oplus z)^{\ast}$, thus proving that $x \bot (y \oplus z)$.

The proof for $\sqcap$ runs along the same lines.
\end{proof}

\begin{defi}
Let $(X,\cdot)$ be a partial semigroup. We define the \emph{universal monoid} $S(X)$ by generators and relations, as:
\[
S(X) := \left< X \, \vert \, xy = z \textnormal{ when } x\cdot y = z \right>_{\textnormal{mon}},
\]
Similarly, we define the \emph{universal group} $G(X)$ as
\[
G(X) := \left< X \, \vert \, xy = z \textnormal{ when } x\cdot y = z \right>_{\textnormal{gr}},
\]
\end{defi}

We will not use Rump's definition of the groups $S(X)$ and $G(X)$ but one that is analogous because of \cite[Theorem 7.2]{Rump_von_Neumann}:

\begin{defi}
The \emph{structure monoid} of an OML $X$ is the universal monoid $S(X)$ where $X$ is equipped with the partial operation $\sqcap$ as given in \autoref{lem:oml_partial_operation}.

The \emph{structure group} of an OML $X$ is the universal group $G(X)$ where $X$ is regarded as a partial semigroup in the same sense.
\end{defi}

\begin{defi}
If $G$ be a group which is equipped with a partial order denoted by $\leq$, we say $\leq$ is a \emph{right-invariant order} or, equivalently, $G$ is \emph{right-ordered}, if for all $x,y,z \in G$ we have the implication
\[
x \leq y \Rightarrow xz \leq yz.
\]
If $G$ is a lattice under $\leq$, we say $G$ is a \emph{right $\ell$-group}, or is \emph{right lattice-ordered}.
\end{defi}

\begin{rema}
\begin{enumerate}[1)]
\item Analogously, one might see $G$ as a regular (i.e. faithfully transitive) right action by automorphisms on the partial order of $G$.
\item If $G$ is a right $\ell$-group, the fact that the right-multiplication maps are automorphisms of the lattice structure on $G$ implies the equations
\begin{align*}
(x \vee y)z & = xz \vee yz \\
(x \wedge y)z & = xz \wedge yz
\end{align*}
for all $x,y,z \in G$. These equations could be taken as an alternative definition of a right $\ell$-group.
\end{enumerate}
\end{rema}

\begin{defi}
Let $G$ be equipped with a right order $\leq$. We define the \emph{positive cone} $G^+$ resp. the \emph{negative cone} $G^-$ by
\begin{align*}
G^+ & := \{g \in G: g \geq e \} \\
G^- & := \{g \in G: g \leq e \}
\end{align*}
where $e$ is the neutral element of $G$.
\end{defi}

The following result is well-known:

\begin{lem} \label{lem:cones_define_orders}
Each positive cone $G^+$ (or negative cone $G^-$) of a right order of a group $G$ is a submonoid of $G$ under the group operation such that $g,g^{-1} \in G^+$ (or $g,g^{-1} \in G^-)$ imply $g = e$.

Conversely, every submonoid $C$ of $G$ fulfilling these conditions is the positive (resp. negative) cone of a right-order on $G$ defined by $g \leq h :\Leftrightarrow hg^{-1} \in C$ (resp. $g \leq h :\Leftrightarrow gh^{-1} \in C$).
\end{lem}

Rump showed that the structure group of any OML $X$ can be given a right lattice order in which $X$ embeds. More precisely:

\begin{thm} \cite[Corollary 4.6]{Rump_von_Neumann}
$G(X)$ has a right lattice-order defined by the negative cone $S(X)$. Furthermore, the natural inclusion $\iota: X \hookrightarrow G(X)$ is an embedding of $X$ as an interval of $G(X)$.
\end{thm}

Furthermore, there is a nice characterization of those structure groups $G(X)$ in terms of special group elements whose properties we will now define.

There are distinguished properties of elements in a right $\ell$-group which we will now describe (see \cite[Section 4]{Rump_von_Neumann} for their prototypes).

\textbf{Caution} is demanded from readers familiar with the original article for it is more comfortable to demand certain properties from negative instead of positive elements, and vice versa.

\begin{defi}
Let $G$ be a right lattice-ordered group. We say an element $\Delta \in G^-$ is
\begin{enumerate}[i)]
\item \emph{singular}\footnote{As opposed to \cite{Rump_von_Neumann}, we decided to define singularity for \emph{negative} elements in $G$ instead of constantly writing down the inverse of a positive singular element. The reader may convince himself that the meaning of the theorems and definitions is not affected by this redefinition.}, if for $x,y \in G^-$ the implication
\[
\Delta \leq xy \Rightarrow yx = x \wedge y
\]
holds,
\item \emph{normal}, if $\Delta(g \vee h) = \Delta g \vee \Delta h$ holds for all $g,h \in G$ (i.e. \emph{left}-multiplication by $\Delta$ is a lattice-automorphism),
\item an \emph{order unit}\footnote{Here we made the same decision as for singularity.} if for all $g \in G$ there is an $m \in \mathbb{Z}$, such that $g \leq \Delta^m$,
\item a \emph{strong order unit} if $\Delta$ is an order unit that is normal.
\end{enumerate}
\end{defi}

\begin{rema}
A remarkable property of a singular element $\Delta$ is that $\Delta \leq xy$ does not only imply $yx = x \wedge y$ but also $xy = x \wedge y$ (see \cite[Definition 4.8, Remark (i)]{Rump_von_Neumann}.

Therefore $\left[ \Delta, e \right]$ is a \emph{commutative} partial semigroup under the partial operation $x \cdot y = xy$ whenever $xy \in \left[ \Delta, e \right]$.

If we \emph{defined} singularity by the implication $\Delta \leq xy \Rightarrow xy = x \wedge y$ there would be no guarantee that $xy = yx$.
\end{rema}

The following theorem of Rump identifies the right $\ell$-groups that can (and will) arise from an OML:

\begin{thm}\cite[Theorem 4.10]{Rump_von_Neumann} \label{thm:identifying_structure_groups}
If $G$ is a lattice-ordered group with a singular strong order unit $\Delta$, then the interval $[\Delta,e]$ is an OML with the complement $g^{\ast} := g^{-1}\Delta$.

Furthermore, the embedding $[\Delta,e] \hookrightarrow G$ identifies $G$ as a structure group for the OML $[\Delta,e]$.

Vice versa, the image of the smallest element $0\in X$ under the natural inclusion $X \hookrightarrow G(X)$ is a singular strong oder unit.
\end{thm}

In the following sections we will work out a description of $G(X)$ whenever $X$ is the OML of $k$-subspaces of a finite-dimensional $k$-vector space equipped with an anisotropic bilinear form.

This description will be explicit, meaning that the elements of $G(X)$ can be identified with elements of a matrix group over a ring of Laurent polynomials.

\section{OMLs of subspaces and paraunitary groups}

In this section, we will give a definition of the \emph{paraunitary group} associated with a $k$-bilinear form.

Let $V$ be a finite-dimensional vector space over a (commutative) field $k$.

Furthermore, let $b:V \times V \to k$ be an \emph{anisotropic}, symmetric $k$-bilinear form, i.e. $b$ is $k$-linear in each argument, $b(v,w) = b(w,v)$ holds for any $v,w \in V$ and we additionally have
\[
b(v,v) = 0 \quad \Leftrightarrow \quad v = 0
\]
(note that anisotropy implies non-degeneracy of $b$ so we do not need to demand this).

\begin{prop} \label{prop:Xb_is_oml}
Let $X$ be the lattice of $k$-subspaces of $V$, with intersection and sum of subspaces being the lattice operations.

The map
\begin{align*}
\ast: X & \to X \\
U & \mapsto U^{\ast} := \left\{ v \in V: b(u,v) = 0 \, \forall u \in U \right\}
\end{align*}
then is an orthocomplementation making $X$ into an OML.
\end{prop}

\begin{defi}
We denote by $X(b)$ the OML of $k$-subspaces of $V$, together with the orthocomplementation given in \autoref{prop:Xb_is_oml}.
\end{defi}

\begin{proof}(of \autoref{prop:Xb_is_oml})

That $U \mapsto U^{\ast}$ is an order-reversing involution whenever $b$ is non-degenerate is shown in \cite[p.346-347]{jacobson1}
, together with the fact that $\dim_k U^{\ast} + \dim_k U = \dim_k V$.

$U \cap U^{\ast}$ must be $0$ all $U \in X$, for any $v \in U \cap U^{\ast}$ fulfils $b(v,v) = 0$ and $b$ is anisotropic. By a dimension argument, $U \oplus U^{\ast} = V$.

Restricting $b$ to subspace $W$ again leaves us with an anisotropic $k$-bilinear form on $W$ which is again non-degenerate\footnote{This would be \emph{wrong} if $b$ was not anisotropic!}.

For any $U\in X$ with $U \subseteq W$ that implies by the same reasoning as above:
\begin{align*}
W & = U \oplus \left\{ v \in W: b(u,v) = 0 \, \forall u \in U \right\} \\
& = U \oplus (U^{\ast} \cap W),
\end{align*}
proving that $X(b)$ is an OML.
\end{proof}

That makes sense of our question how to describe $G(X(b))$, i.e. the structure group of the OML $X(b)$.

The non-degeneracy of $b$ implies that each $k$-linear $\varphi:V \to V$ has a unique \emph{left-adjoint}, i.e. there is a $k$-linear mapping $\varphi^{\prime}:V\to V$ fulfilling
\[
b(\varphi^{\prime}(v),w) = b(v,\varphi(w))
\]
for all $v,w \in V$.

We define the \emph{unitary group} $U(b)$ associated with $b$ to be the group of all $k$-linear maps $\varphi: V \to V$ such that there holds
\[
b(\varphi(v), \varphi(w)) = b(v,w)
\]
for all $v,w\in V$ or, equivalently, $\varphi^{\prime}\varphi = 1$.

We will now extend $b$ as follows:

Let $-^{\ast}: k\laur \to k\laur$ be the unique morphism of $k$-algebras that sends $t$ to $t^{-1}$. Clearly, $-^{\ast}$ is an involution of $k\laur$, i.e. $f^{\ast \ast} = f$ holds for all $f \in k\laur$.

Additionally, we extend $V$ to $V\laur := k\laur \otimes_k V$. Using coordinates, $V\laur$ can be identified with $k\laur^n$.

Clearly, the elements of $k\laur^n$ can also be written as
\[
v = \sum_{i = -\infty}^{\infty} t^i v_i
\]
with $v_i \in V$, and $v_i = 0$ for all but finitely many indices $i$. We will mean $v_i$ when speaking of the $i$-th \emph{coordinate} of $v$.

We now extend $b$ to $V\laur$ by the rule
\begin{equation} \label{eq:def_of_b_tilde}
\tilde{b}(v,w) := \sum_{n=-\infty}^{\infty} t^n \cdot \left( \sum_{i = -\infty}^{\infty} b(v_{i-n},w_i)\right).
\end{equation}
This form is additive in both arguments, $k\laur$-linear in the second and $k\laur$-antilinear in the first, i.e. we have
\[
\tilde{b}(f v,g w) = f^{\ast}g\cdot \tilde{b}(v,w).
\]
for $f,g \in k \laur$, $v,w \in V \laur$.

This can be seen easiest as follows: Clearly, $\tilde{b}(t^iv_i,t^jw_j) = t^{j-i}b(v_i,w_j)$. This implies
\[
\tilde{b}(t^m t^i v_i, t^n t^j w_j) = t^{n-m}t^{j-i} \cdot b(v_i,w_j) = (t^m)^{\ast}t^n \cdot \tilde{b}(t^iv_i,t^jw_j).
\]
which can then be extended by bilinearity to arbitrary elements of $k\laur$ resp. $V \laur$.

For elements of $k\laur$ and $V \laur$ we can specialize $t$ to arbitrary values in $k \setminus \{ 0 \}$ in an obvious manner in order to get elements in $k$ resp. $V$.

It is an important feature of $\tilde{b}$ that it specializes to $b$ in the following sense:

\begin{lem}
For $v,w \in V\laur$ holds
\begin{equation} \label{eq:specializing_b}
\left( \tilde{b}(v,w) \right)(1) = b(v(1),w(1))
\end{equation}
\end{lem}

\begin{proof}
For $v = t^iv_i$, $w = t^jw_j$ we have
\[
\left( \tilde{b}(v,w) \right)(1) = 1^{j-i}b(v_i,w_j) = b(v(1),w(1)).
\]
and this extends bilinearly to all elements of $V\laur$.
\end{proof}

\begin{lem} \label{lem:tilde_b_anisotropic}
$\tilde{b}$ is anisotropic, i.e. we have
\[
\tilde{b}(v,v) = 0 \Leftrightarrow v = 0
\]
where $v \in V\laur$.
\end{lem}

\begin{proof}
Let $v = v(t) \in k\laur \setminus \{0 \}$ fulfil $\tilde{b}(v(t),v(t)) = 0$.

We might assume that $v(t) \in k\polyp^n$ and that the entries of $v(t)$ have no common zero - this can be achieved by multiplication by a suitable polynomial, using antilinearity. Using \eqref{eq:specializing_b}, we get
\[
b(v(1),v(1)) = 0
\]
which implies $v(1) = 0$. Therefore $t=1$ is a common zero of all entries, contrary to our assumption.
\end{proof}

We will now continue the assignment $\varphi \mapsto \varphi^{\prime}$ in a meaningful way to $\textnormal{End}_{k\laur} V\laur$. Note that each element $\varphi$ in the latter ring can be written as a Laurent series of elements in $\textnormal{End}_kV$, i.e. 
\[
\varphi = \sum_{i = -\infty}^{\infty} t^{i} \varphi_i
\]
where $\varphi_i \in \textnormal{End}_k V$.

This will later allow us to specialize $k\laur$-endomorphisms of $V\laur$.

\begin{prop} \label{prop:para_adjoint}
For each $\varphi \in \textnormal{End}_{k\laur} V\laur$ there is a unique $\varphi^{\prime} \in \textnormal{End}_{k\laur} V\laur$ fulfilling
\begin{equation} \label{eq:adjoint_b_tilde}
\tilde{b}(\varphi^{\prime}(v),w) = \tilde{b}(v,\varphi(w))
\end{equation}
which is given by
\begin{equation} \label{eq:adjoint_b_tilde2}
\varphi^{\prime} = \sum_{i=-\infty}^{\infty} t^{-i} \varphi_i^{\prime}.
\end{equation}
\end{prop}

\begin{proof}
\autoref{lem:tilde_b_anisotropic} implies that $\tilde{b}$ is non-degenerate. The uniqueness of a left-adjoint follows easily from that.

That the given element indeed is the correct choice already follows (by linearity) from showing validity on generators, i.e. for $v = t^i v_i$, $w = t^jw_j$, $\varphi = t^k \varphi_k$:
\[
\tilde{b}(t^{-k}\varphi_k^{\prime}t^iv_i,t^jw_j) = t^{j-i+k} b(\varphi_k^{\prime}v_i,w_j) = t^{j+k-i} b(v_i,\varphi_k w_j) = \tilde{b}(t^iv_i,t^k t^jw_j).
\]
\end{proof}

We can now make the first central definition of this section
\begin{defi}
The \emph{paraunitary group} $\PU(b)$ associated with $b$ is the group of invertible $k\laur$-linear endomorphisms $\varphi: V\laur \to V\laur$ such that
\begin{equation} \label{defi:para_group}
\tilde{b}(\varphi(v),\varphi(w)) = \tilde{b}(v,w)
\end{equation}
holds for all $v,w \in V\laur$.
\end{defi}

Rewriting \eqref{defi:para_group} and using \autoref{prop:para_adjoint}, one gets another description of $\PU(b)$:

\begin{cor} \label{cor:para_matrix}
The group $\PU(b)$ is given by all $\varphi \in \textnormal{End}_{k\laur}V\laur$ such that
\begin{equation} \label{eq:para_matrix}
\varphi^{\prime}\varphi = 1.
\end{equation}
\end{cor}

It is easily seen that mapping $\varphi$ to $\varphi(1)$ is multiplicative, therefore we get that the \emph{canonical specialization}
\begin{align*}
\varepsilon_1: \PU(b) & \to U(b) \\
\varphi & \mapsto \varphi(1)
\end{align*}
which is well-defined due to \eqref{eq:specializing_b} and the easy to see identity $(\varphi(v))(1) = \varphi(1)(v(1))$.

Consequently, $\varepsilon_1$ is a homomorphism of groups.

There is also a \emph{canonical inclusion} $\iota: U(b) \to \PU(b)$, given by sending $\varphi \in U(b)$ to $t^0 \varphi \in \PU(b)$. We will refer to the image of $\iota$ as the \emph{constants}.

It is easily seen that $\varepsilon_1 \circ \iota = \text{id}_{U(b)}$. 

We can now define the group which is relevant to our further investigations:
\begin{defi}
The \emph{pure paraunitary group} associated with $b$ is defined as the subgroup
\[
\PPU(b):= \ker \varepsilon_1 \unlhd \PU(b).
\]
\end{defi}

We have seen that $\varepsilon_1 \circ \iota = \text{id}_{U(b)}$ which implies the following semidirect decomposition:
\begin{cor} \label{cor:semidirect_product}
$\PU(b) = \PPU(b) \rtimes U(b)$ where $U(b)$ is identified with the matrices with the constants in $\PU(b)$.
\end{cor}

\begin{rema}
We called these matrices \emph{pure} in order to suggest an analogy to the definition of the pure braid group $P_n$ (\cite[Section 1.3]{Kassel_braid}) as the kernel of a specialization map
\[
\varepsilon: B_n \to S_n
\]
where $B_n$ is the braid group on $n$ strands and $S_n$ is the symmetric group on $n$ letters.

We will see later that $\PPU(b)$ might be an even better counterpart of $B_n$.
\end{rema}

\section{A right-invariant lattice-order on \texorpdfstring{$\PPU(b)$}{PPU(b)}}

In what follows, we will denote by $M_n(R)$ the ring of $n \times n$-matrices over a ring $R$ or, equivalently, $\textnormal{End}_RR^n$.

We can easily define a right-invariant partial order on $\PPU(b)$:

We therefore define $\PPU(b)^+$ resp. $\PPU(b)^-$ as the submonoid of elements in $\PPU(b)$ which are contained in $M_n(k\polyp)$ resp. $M_n(k\polym)$. Equivalently, these are the elements $\varphi \in \PPU(b)$ whose Laurent expansion has $\varphi_i = 0$ for $i < 0$ resp. $i > 0$.

\begin{prop} \label{prop:PPU+_is_cone}
$\PPU(b)^+$ and $\PPU(b)^-$ are the submonoids of positive resp. negative elements of a right-invariant partial order on $\PPU(b)$ given by
\[
\psi \geq \varphi :\Leftrightarrow \psi \varphi^{-1} \in \PPU(b)^+ \Leftrightarrow \varphi \psi^{-1} \in \PPU(b)^-
\]
\end{prop}

\begin{proof}
Let $\varphi \in \PPU(b)^+$. Equation \eqref{defi:para_group} tells us $\varphi^{-1} = \varphi^{\prime}$ and $\varphi^{\prime} \in \PPU(b)^-$ follows from \eqref{eq:adjoint_b_tilde2}.

Therefore, if $\varphi,\varphi^{-1}$ both lie in $\PPU(b)^+$, we have $\varphi \in \PPU(b)^+ \cap \PPU(b)^-$, i.e. it is of the form $t^0\varphi_0$.

But $t^0 \varphi_0 \in \PPU(b)$ forces $\varphi_0$ to be $1$, showing that $\PPU(b)^+ \cap (\PPU(b)^+)^{-1} = \{ 1 \}$.

Furthermore, $\PPU(b)^+$ is a subsemigroup of $\PPU(b)$ and therefore (by \autoref{lem:cones_define_orders}) defines a right-invari\-ant order of $\PPU(b)$.
\end{proof}

The ultimate goal of this section is to prove the following
\begin{thm} \label{thm:PPU_is_lattice}
The right-invariant order on $\PPU(b)$ given by the positive cone $\PPU(b)^+$ is a lattice-order.
\end{thm}

\subsection*{Step I: Introducing \texorpdfstring{$\Omega$}{Omega}}~

By restriction of scalars, $V\laur$ can not only be seen as a $k\laur$-module but also as a $k\polym$-module.

$V\laur$ contains the $k\polym$-submodule $V\polym := k\polym \otimes_k V$ which, using coordinates, can be identified with the elements $v \in V\laur$ of the form
\[
v = \sum_{i = - \infty}^0 t^i v_i,
\]
i.e. those whose Laurent series has only non-positive exponents.

The definition of the $k\polyp$-submodule $V\polyp$ of $V\laur$ will be similar to the definition of $V\polym$. $V\polyp$ will play an important role later.

We now \glqq cut off\grqq\ any terms with exponents smaller or equal to zero by defining the $k\polym$-module
\[
V^{\oplus} := V\laur / V\polym.
\]

Each $v \in V^\oplus$ now has a \emph{unique} representative of the form
\[
v = \sum_{i = 1}^{\infty} t^iv_i
\]
(with only finitely many $v_i \neq 0$).

Using these expressions, multiplications by matrices with entries in $k\laur$ can be performed by first calculating the result in $V\laur$ and cutting of any terms with negative exponents in the result.

We define $\Sub(V^{\oplus})$ as the lattice of \emph{finitely generated} $k\polym$-submodules of $V^{\oplus}$, ordered by inclusion.

That $\Sub(V^{\oplus})$ is closed under finite sums is clear. That intersections are also finitely generated follows from the following lemma (that will also be crucial later):
\begin{lem} \label{lem:subs_are_findim}
Each $U \in \Sub(V^{\oplus})$ is finite-dimensional when regarded as a vector space over $k$.
\end{lem}

\begin{proof}
Let $U$ be generated by the elements $v^j$ ($1 \leq j \leq l$). We can represent these by
\[
v^j = \sum_{i = 1}^{m_j} t^i v_i^j.
\]
Setting $m := \max \{m_j:1\leq j \leq l \}$ we see that all $v^j$ are contained in the $k\polym$-submodule consisting of all elements of the form
\[
v = \sum_{i=1}^m t^i v_i
\]
(note that multiplication by elements of $k\polym$ can not increase the degrees in any such expression).

Recalling that $\dim_kV = n$ we see that this submodule has $k$-dimension $mn$.
\end{proof}

Clearly, $V\polym$ can be seen as a left $M_n(k\polym)$-module in a natural way. Therefore, the elements of $\PPU(b)^-$ act $k\polym$-linearly on $V\polym$.

An important property of the left-multiplication maps is the following

\begin{prop}
For each $p \in \PPU(b)^-$, the map $v \mapsto pv$ of $V^{\oplus}$ is surjective:
\end{prop}

\begin{proof}
The map $v \mapsto pv$ is clearly surjective on $V\laur$ (for we have $p^{-1} = p^{\prime}$). This descends to the factor $V^{\oplus}$.
\end{proof}

We now define a map:
\begin{align*}
\Omega: \PPU(b)^- & \to \Sub(V^{\oplus}) \\
\varphi & \mapsto \ker \varphi.
\end{align*}

An immediate consequence of this definition is:

\begin{cor}
Let $\PPU(b)^-$ be equipped with the order given by right-divisibility. Then $\Omega$ is an order-preserving map.
\end{cor}

\begin{proof}
This follows from the observation that
\[
\Omega(\psi \varphi) = \ker(\psi \varphi) \supseteq \ker(\varphi) = \Omega(\varphi).
\]
\end{proof}

In the following sections, we will prove something stronger, i.e.
\begin{thm} \label{thm:omega_is_iso}
$\Omega$ is an isomorphism of posets. In particular, $\PPU(b)^-$ is a modular lattice with respect to right divisibility.
\end{thm}

\subsection*{Step II: A \texorpdfstring{$k$}{k}-bilinear form on \texorpdfstring{$V\laur$}{V[t,t -1]}}~

For $v,w \in V\laur$, we define the symmetric $k$-bilinear form
\[
\tilde{b}_0(v,w) := \sum_{i = -\infty}^{\infty}b(v_i,w_i)
\]
which is the $0$-th coefficient of $\tilde{b}(v,w)$.

We have the orthogonal decomposition\footnote{Caution is demanded for we do \emph{not} have that all $k$-subspaces $U$ have an orthogonal complement with respect to $\tilde{b}_0$ because $\tilde{b}_0$ is not always anisotropic, despite $\tilde{b}$ being anisotropic!}
\[
V\laur = t V\polyp \oplus_{\tilde{b}_0} V\polym
\]
(the subindex meaning that orthogonality is considered with respect to $\tilde{b}_0$).

Note that $V\polym$ \emph{is} indeed the orthogonal complement of $t V\polyp$ with respect to $\tilde{b}_0$, i.e. there are not more orthogonal elements: for any $v \in V\laur$ with $v_i \neq 0$ for some $i > 0$ we have $\tilde{b}_0(t^iv_i,v) = b(v_i,v_i) \neq 0$, while $t^iv_i \in t V\polyp$.

$\tilde{b}_0$ is furthermore preserved by any $\varphi \in \PPU(b)$ (for those maps preserve \emph{all} coefficients of $\tilde{b}$!). Thus we have proved that
\begin{lem} \label{lem:decomp_of_images}
For all $\varphi \in \PPU(b)$ we have
\[
V\laur = \varphi (tV\polyp) \oplus_{\tilde{b}_0} \varphi(V\polym).
\]
\end{lem}

\subsection*{Step III: \texorpdfstring{$\Omega$}{Omega} is surjective}~

We define the $k\polym$-submodule
\[
V^1 := \left\{tv: v\in V \right\} \subseteq V^{\oplus}.
\]
It turns out that this module is essential in $V^{\oplus}$:
\begin{lem} \label{lem:v0_essential}
If $0 \neq U \subseteq V^{\oplus}$ is a $k\polym$-submodule then $U \cap V^1 \neq 0$.
\end{lem}

\begin{proof}
Let $0 \neq v \in U$ and write
\[
v = \sum_{i=1}^m t^iv_i
\]
where $v_m \neq 0$. Then
\[
0 \neq t v_m = t^{-m+1}v \in U \cap V^1.
\]
\end{proof}

For any $k$-subspace $U \subseteq V$ we define the orthogonal projection onto $U$ by:
\begin{align*}
\pi_U: V = U \oplus U^{\ast} & \to V \\
u + u^{\prime} & \mapsto u \\ 
\end{align*}
where $U^{\ast} = \{ v \in V:b(u,v) = 0 \, \forall u \in U \}$. Note that we have $U \oplus U^{\ast} = V$ for all subspaces $U$ because $b$ is anisotropic. Additionally, $U^{\ast \ast} = U$.

It is clear that we have $\pi_U^2 = \pi_U$. Furthermore, $\pi_U$ is self-adjoint with respect to $b$, for if we write $u = u_1 + u_2$ and $v = v_1 + v_2$ with $u_1,v_1 \in U$, $u_2,v_2 \in U^{\ast}$, we have
\[
b(\pi_U(u),v) = b(u_1,v_1+v_2) = b(u_1,v_1) = b(u_1+u_2,v_1) = b(u,\pi_U(v)),
\]
i.e. $\pi_U^{\prime} = \pi_U$.

We collect - without proof - some well-known facts on these projection operators:

\begin{lem} \label{lem:identities_for_projections}
Let $U,W \subseteq V$ be $k$-subspaces. Then the following identities hold:
\begin{enumerate}[a)]
\item $\pi_{U \oplus W} = \pi_U + \pi_W$ whenever $U \bot W $. Particularly, $\pi_{U} + \pi_{U^{\ast}} = 1$.
\item $\pi_{U \sqcap W} = \pi_U \pi_V$ whenever $U \top W$.\footnote{But it is well-known that this condition is not necessary for $\pi_{U \cap W} = \pi_U \pi_W$ to hold.}
\item We have $\pi_U \pi_W = 0$ exactly when $U \bot W$.
\end{enumerate}
\end{lem}

We now set
\[
p_U := t^{-1}\pi_{U^{\ast}} + \pi_U. \label{definition_of_pv}
\]

\begin{lem} \label{lem:props_of_pv}
We have the following properties:
\begin{enumerate}[i)]
\item $p_U \in \PPU(b)^-$ for all $k$-subspaces $U\subseteq V$,
\item $\Omega(p_U) = \{ tv: v \in U^{\ast} \}$.
\end{enumerate}
\end{lem}

\begin{proof}
\begin{enumerate}[i)]
\item We calculate
\begin{align*}
p_U^{\prime} p_U & = (t\pi_{U^{\ast}}^{\prime} + \pi_U^{\prime})(t^{-1}\pi_{U^{\ast}} + \pi_U) \\
& = (t\pi_{U^{\ast}} + \pi_U)(t^{-1}\pi_{U^{\ast}} + \pi_U) \\
& = t\underbrace{\pi_{U^{\ast}}\pi_U}_{=0} + \pi_{U^{\ast}}^2 + \pi_U^2 + t^{-1} \underbrace{\pi_U \pi_{U^{\ast}}}_{=0} \\
& = \pi_{U^{\ast}} + \pi_U = 1.
\end{align*}
and $p_U(1) = \pi_{U^{\ast}} + \pi_U = 1$. Therefore, $p_U \in PPU(b)^-$.
\item Let $v \in V^{\oplus}$ fulfil $p_Uv = 0$. Let this element be represented as
\[
v = \sum_{i=1}^{\infty} t^i v_i.
\]
Then one calculates that $p_Uv$ is represented by
\[
p_Uv = \sum_{i=1}^{\infty} t^i(\pi_U(v_i) + \pi_{U^{\ast}}(v_{i+1}))
\]
ignoring the summand $t^0 \pi_{U^{\ast}}(v_1)$ which is in $V\polym$.

We need to have $\pi_U(v_i) + \pi_{U^{\ast}}(v_{i+1}) = 0$ for all $i \geq 1$.

For we have $U \cap U^{\ast} = 0$, this implies $\pi_U(v_i) = 0$ and $\pi_{U^{\ast}}(v_{i+1})$. We infer that $v_i \in U^{\ast}$ and $v_{i+1} \in U^{\ast \ast} = U$ for $i \geq 1$.

But then $v_i \in U \cap U^{\ast}$ must hold for $i \geq 2$, i.e. $v_i = 0$.

We have therefore shown that $v_1 \in U^{\ast}$ and $v_i = 0$ for all $i > 1$. Therefore $\Omega(p_U) \subseteq \{tv:v \in U^{\ast} \}$.

On the other hand, if $v \in U^{\ast}$, one calculates easily that $p_U \cdot tv = t^0 v_1$ which is zero in $V^{\oplus}$, thus showing the other inclusion.
\end{enumerate}
\end{proof}

We are now able to prove the following
\begin{prop} \label{prop:omega_is_sur}
$\Omega$ is surjective.
\end{prop}

\begin{proof}
Given $U \in \Sub(V^{\oplus})$, we must find a $\varphi \in \PPU(b)$ with $\ker \varphi = U$. This will be done by induction over $m := \dim_kU$ (which is possible because of \autoref{lem:subs_are_findim}).

For $m = 0$, $\varphi$ can be chosen to be the identity.

If $m > 0$, by \autoref{lem:v0_essential}, $U \cap V^1 \neq 0$. Therefore, $W := \{v \in V:tw \in U \} \neq 0$. Now assume that we have shown that each submodule of dimension smaller than $m$ is the kernel of some $\varphi \in \PPU(b)^-$.

We have $\ker(p_{W^{\ast}}) = \ker(p_{W^{\ast}}) \cap U = tW^{\ast \ast} = tW$. Therefore, $\dim_k \left(p_{W^{\ast}}(U) \right) < m$.

Let $\varphi \in \PPU(b)^-$ be such that $\ker(\varphi) = p_{W^{\ast}}(U)$. Then
\[
\ker(p_{W^{\ast}} \circ \varphi) = p_{W^{\ast}}^{-1} \left( \ker \varphi \right) = p_{W^{\ast}}^{-1}\left(p_{W^{\ast}}(U)\right) = U
\]
where the last equality follows from \autoref{lem:props_of_pv}.
\end{proof}

\subsection*{Step IV: \texorpdfstring{$\Omega$}{Omega} is injective}~

In the following lemma, we denote by $\pi: V\laur \to V^{\oplus}$ the canonical projection map.

\begin{lem} \label{lem:kernel_image_trick}
For $\varphi \in \PPU(b)^-$ holds
\[
V\laur = \pi^{-1}(\Omega(\varphi)) \oplus_{\tilde{b}_0} \varphi^{-1} t V\polyp.
\]
\end{lem}

\begin{proof}
For $v \in V\laur$ we have the equivalences
\[
\pi^{-1}(\Omega(\varphi)) \Leftrightarrow \varphi(v) \in V\polym \Leftrightarrow v \in \varphi^{-1}(V\polym).
\]
On the other hand, by \autoref{lem:decomp_of_images}, we have
\[
V\laur = \varphi^{-1}t V\polyp \oplus_{\tilde{b}_0} \varphi^{-1}V\polym,
\]
thus proving the lemma.
\end{proof}

\begin{prop} \label{prop:omega_is_inj}
$\Omega$ is injective.
\end{prop}

\begin{proof}
Let $\Omega(\varphi_1) = \Omega(\varphi_2)$. By \autoref{lem:kernel_image_trick}, this is equivalent to $\varphi_1^{-1}V\polyp = \varphi_2^{-1}tV\polyp$ resp. $\varphi_2 \varphi_1^{-1}tV\polyp = tV\polyp$ resp. $\varphi_1 \varphi_2^{-1}V\polyp = V\polyp$.

Therefore, both $\varphi_2 \varphi_1^{-1}$ as $(\varphi_2 \varphi_1^{-1})^{-1}$ lie in $\PPU(b)^+$.  By \autoref{prop:PPU+_is_cone} this is only possible when $(\varphi_2 \varphi_1^{-1})^{-1} = 1$ resp. $\varphi_1 = \varphi_2$.
\end{proof}

We can finally give the desired proof of \autoref{thm:omega_is_iso}:
\begin{proof}
By \autoref{prop:omega_is_sur} and \autoref{prop:omega_is_inj}, we have that $\Omega$ is a bijection. It remains to show that $\varphi_1$ right-divides $\varphi_2$ whenever $\Omega(\varphi_1) \subseteq \Omega(\varphi_2)$:

By \autoref{prop:omega_is_sur}, we have $\varphi_1(\Omega(\varphi_2)) = \Omega(\psi)$ for some $\psi \in \PPU(b)^-$.

Therefore
\[
\Omega(\psi \varphi_1) = \varphi_1^{-1}(\Omega(\psi)) = \varphi_1^{-1}(\varphi_1(\Omega(\varphi_2))) = \Omega(\varphi_2)
\]
because $\varphi_1$ is surjective. \autoref{prop:omega_is_inj} now implies that $\varphi_2 = \psi \varphi_1$.
\end{proof}

\subsection*{Step V: Going over to \texorpdfstring{$\PPU(b)$}{PPU(b)}}~

We can now reformulate \autoref{thm:omega_is_iso} as follows:

\begin{cor}
Mapping $\varphi$ to $\varphi^{-1}V\polym$ defines an isomorphism of lattices between
\begin{itemize}
\item $\PPU(b)^-$ with the right-divisibility order, and
\item finitely generated $k\polym$-submodules $U \subseteq V\laur$ containing $V\polym$, ordered by inclusion.
\end{itemize}
\end{cor}

We can now deduce the following lemma which immediately implies \autoref{thm:PPU_is_lattice}:

\begin{lem}
Mapping $\varphi$ to $\varphi^{-1}V\polym$ establishes an isomorphism of posets between:
\begin{itemize}
\item $\PPU(b)$, with the right-divisibility order given by $\PPU(b)^-$,
\item finitely generated $k\polym$-submodules $U \subseteq V\laur$ containing $t^m V\polym$ for some $m \in \mathbb{Z}$, ordered by inclusion.
\end{itemize}
\end{lem}

\begin{proof}
Let $U \leq V\laur$ contain $t^m V \polym$. Then $t^{-m}U$ contains $V\polym$. Therefore, $t^{-m}U = \varphi^{-1} V\polym$ for a $\varphi \in \PPU(b)^-$, showing that $U = (\varphi t^{-m})^{-1} V\polym$, so the given map is surjective.

Injectivity follows from exactly the same argument as in the proof of \autoref{prop:omega_is_inj}.

In order to show that $\varphi \mapsto \varphi^{-1}V \polym$ is an order-embedding it needs to be shown that $\varphi \in \PPU(b)^-$ iff $\varphi^{-1} V\polym \subseteq V\polym$. But this is the case exactly when $V \polym \subseteq V \polym$ and this happens if and only $\varphi \in \PPU(b)^-$.
\end{proof}

\section{\texorpdfstring{$\PPU(b)$}{PPU(b)} is a structure group}

The reader is reminded of the fact that $\PPU(b)^-$ is the negative cone of $\PPU(b)$. Note that the right-invariant order induced on $PPU(b)^-$ is \emph{opposite} to the right-divisibility order on $\PPU(b)^-$ we used in the preceding section!

In what follows, $\PPU(b)$ carries its right-invariant order defined by the negative cone $\PPU(b)^-$ resp. the positive cone $\PPU(b)^+$.

Before tackling our proof that $\PPU(b)$ can be identified with the structure group $G(X(b))$ we collect a few lemmata:

The structure of the interval $\left[t^{-1},1 \right]$ is given as follows:

\begin{lem} \label{lem:Delta_1_consists_of}
$\left[t^{-1},1 \right]$ consists exactly of the elements $p_U$, as defined on page \pageref{definition_of_pv}.
\end{lem}

\begin{proof}
\autoref{lem:props_of_pv} and \autoref{thm:omega_is_iso} show that $\Omega$ establishes a bijection between the $k\polym$-submodules of $V^{\oplus}$ that lie between $tV = \Omega(p_0)$ and $0 = \Omega(p_V)$.

Each of those submodules must be of the form $tU$ for some $k$-subspace $U \subseteq V$, and each such subspace arises as $\Omega(p_{U^{\ast}})$.

Noting that $p_0 = 1$ and $p_V = t^{-1}$ proves the lemma.
\end{proof}

The following lemma shows that some pairs of elements $p_U \in \PPU(b)^-$ behave pretty well under multiplication, i.e.:

\begin{lem} \label{lem:mult_of_p_u}
If $U \top W$ are subspaces of $V$, we have $p_U p_W = p_{U \sqcap W}$.

If $U,W \subseteq V$ are arbitrary subspaces then $p_U p_W \in \left[t^{-1},1 \right]$ holds if and only iff $U \top W$.
\end{lem}

\begin{proof}
For the first part, we just calculate
\begin{align*}
p_U p_W & = (t^{-1}\pi_{U^{\ast}} + \pi_U)(t^{-1}\pi_{W^{\ast}} + \pi_W) \\
& = t^{-2}\pi_{U^{\ast}}\pi_{W ^{\ast}} + t^{-1}(\pi_{U^{\ast}}\pi_W + \pi_U \pi_{W^{\ast}}) + \pi_U \pi_W \\
& = t^{-2}\cdot 0 + t^{-1}(\pi_{U^{\ast}} + \pi_{W^{\ast}}) + \pi_{U \sqcap W} \quad \textnormal{(\autoref{lem:identities_for_projections})}\\
& = t^{-1}(\pi_{U^{\ast} \oplus W^{\ast}}) + \pi_{U \sqcap W} \\
& = t^{-1} \pi_{(U \sqcap W)^{\ast}} + \pi_{U \sqcap W} = p_{U \sqcap W}.
\end{align*}
Looking at the third line of this calculation and taking \autoref{lem:Delta_1_consists_of} into account, we infer that $p_U p_W \in \left[t^{-1},1 \right]$ implies $\pi_{U^{\ast}}\pi_{W^{\ast}} = 0$. By \autoref{lem:identities_for_projections}, this is exactly the case when $U^{\ast} \bot W^{\ast}$ or, equivalently, $U \top W$.

The other implication follows immediately from the first part of the lemma.
\end{proof}

\begin{lem} \label{lem:Xb_is_embedded}
The map
\begin{align*}
\iota: X(b) & \to \left[t^{-1},1 \right] \\
U & \mapsto p_U 
\end{align*}
is an isomorphism of lattices (with\emph{out} respect to the OML structures).
\end{lem}

\begin{proof}
From \autoref{lem:Delta_1_consists_of} the surjectivity follows. Furthermore, $\iota$ is order-preserving: if $W \subseteq U \subseteq V$, we have $W = U \sqcap (U^{\ast} + W)$. Setting $Z := U^{\ast} + W$, we therefore get by \autoref{lem:mult_of_p_u} that $p_W = p_{Z \sqcap U} = p_Z p_U$, implying $p_W \leq p_U$.

We will now construct an inverse to $\iota$:

Using the obvious isomorphism $\gamma: \left[ 0, tV \right] \overset{\sim}{\longrightarrow} X(b)$ taking $tU$ to $U$ we get that
\begin{align*}
\left( \gamma(\Omega(p_U)) \right)^{\ast} & = \left( \gamma(tU^{\ast}) \right)^{\ast} \quad \textnormal{(\autoref{lem:props_of_pv})} \\
& = U^{\ast \ast} = U.
\end{align*}
Thus we have constructed an order-preserving inverse to $\iota$.
\end{proof}

We now show the second main result of this article:

\begin{thm}
$\PPU(b)$ is a structure group for the OML $X(b)$. 
$X(b)$ can be identified with the interval $[t^{-1},1]$ in the order given by the negative cone $\PPU(b)^-$, and the orthocomplementation on this interval is given by $g \mapsto t^{-1}g^{-1}$.
\end{thm}

\begin{proof}
First, we show that $t^{-1}$ is a singular strong order unit in $\PPU(b)$:

Left-multiplication by $t^{-1}$ is order-preserving because $t^{-1}$ is central in $\PPU(b)$, therefore left- and right-multiplication by $t^{-1}$ coincide, the latter being order-preserving by the mere fact that $\PPU(b)$ is right-ordered. Therefore, $t^{-1}$ is normal.

Let $p_U,p_W \in \left[ t^{-1},1 \right]$ fulfil $p_U p_W \in \left[ t^{-1},1 \right]$. By \autoref{lem:mult_of_p_u}, this implies $U \top W$.

But then \autoref{lem:mult_of_p_u} implies that $p_W p_U = p_{W \sqcap U} = p_W \wedge p_U$ the latter equality following from \autoref{lem:Xb_is_embedded}. Hence $t^{-1}$ is singular.

It remains to show that each $p \in \PPU(b)$ is majorized by some $(t^{-1})^m$:

We have $p^{\prime}p = 1$. Writing $p^{\prime}$ as matrices, it becomes clear that $t^k p^{\prime} \in \PPU(b)^+$ for some large enough $k$. Therefore, $(t^k p^{\prime})p = t^k$, implying that $p \leq t^k = (t^{-1})^{-k}$.

We have thus shown that $t^{-1}$ is a singular strong order unit of $\PPU(b)$.

By \autoref{thm:identifying_structure_groups}, we conclude that $\PPU(b)$, with its given lattice order, can be identified as the structure group of the OML $\left[t^{-1},1 \right]$ whose orthocomplementation is given by $p_U \mapsto (p_U)^{\ast} := p_U^{-1}t^{-1}$.

By \autoref{lem:mult_of_p_u}, we have $p_Up_{U^{\ast}} = t^{-1}$. Thus $(p_U)^{\ast} = p_{U^{\ast}}$, i.e. $\left[ t^{-1},1 \right]$ and $X(b)$ are isomorphic as OMLs.
\end{proof}

\begin{rema}
The reader may have recognized that all proofs and constructions can be put into the bigger framework of \emph{hermitean} anisotropic forms. We give a brief outline:

If $\sigma: k \to k,\, x \mapsto x^{\sigma}$ is an involutive field automorphism and $V$ is a $k$-vector space we call $b: V \times V \to V$ \emph{hermitean} with respect to $\sigma$, if $b$ is additive in both arguments, antilinear in the first and linear in the second argument, i.e. $b(r\cdot v,s\cdot w) = r^{\sigma}s \cdot b(v,w)$ for all $r,s \in k$, $v,w \in V$.

The definition of anisotropy is the same as in the bilinear case.

One proceeds to define $V\laur$ as before and define $\tilde{b}$ on $V\laur \times V\laur$ by Equation \eqref{eq:def_of_b_tilde}. The only difference is that $\tilde{b}$ is that $\tilde{b}$ is hermitean with respect to the involution on $k\laur$ given by:
\begin{align*}
\ast: k \laur & \to k \laur \\
\sum_{i = - \infty}^{\infty} t^i x_i & = \sum_{i = - \infty}^{\infty} t^{-i} x_i^{\sigma}.
\end{align*}
The definition of the paraunitary resp. pure paraunitary group continue as in the symmetric case, and so do all results up to this point.
\end{rema}

\section{Conclusions}

\subsection{\texorpdfstring{$\PPU(b)$}{PPU(b)} as a braid group for \texorpdfstring{$U(b)$}{U(b)}}

Recall that we got $\PPU(b)$ as the kernel of the specialization map $\varepsilon_1:\PU(b) \to U(b)$.

Imitating the proof of equation \eqref{eq:specializing_b} one sees that there is exactly one other specialization map $\PU(b) \to U(b)$ given by setting $t$ to $-1$, and the image of a the generator $p_U \in \PPU(b)$ is given by
\[
p_U(-1) = -\pi_{U^{\ast}} + \pi_U = 1 - 2\pi_U
\]
which is the reflection of $V$ (with respect to $b$) at the subspace $U^{\ast}$. By the Cartan-Dieudonné-theorem (see \cite[Chapter 1, §7]{lam_quadratic}) these reflections generate $U(b)$.

A look at \autoref{lem:mult_of_p_u} makes clear that the atoms $p_{kv}$ ($v \in V \setminus \{ 0 \}$) which correspond to the hyperplane reflections in the Cartan-Dieudonné-theorem are already a set of generators for $\PPU(b)$.

We can furthermore calculate:
\[
(p_U - t^{-1})(p_U - 1) = (-t^{-1}\pi_U + \pi_U)(t^{-1}\pi_{U^{\ast}} - \pi_{U^{\ast}}) = - (t^{-1} - 1)^2 \pi_U \pi_{U^{\ast}} = 0
\]
which is - up to some affine substitutions - the quadratic relation fulfilled by the generators of the Hecke algebra of an Artin-Tits group (see \cite[Chapter 1, Section 2]{mathas}).

The existence of the Garside element $t^{-1}$, the generation by preimages of reflections in $U(b)$ together with the Hecke relations for the generators, all of these suggest to see $\PPU(b)$ as a braid group for the unitary group!

We use this opportunity to pose the following
\begin{ques*}
Describe the quotient of the group algebra $k\left[ \PPU(b) \right]$ by the Hecke relations $(p_U - t^{-1})(p_U - 1) = 0$.
\end{ques*}

The description of $PPU(b)$ as a matrix group preserving a unitary bilinear form suggests a further analogy to the braid groups, whose representing matrices under the Burau representation act by unitary transformations with respect to a very similar bilinear form (see \cite{Squier_Burau}).

\subsection{Factorization of paraunitary matrices}

Using further results of Rump, we get a very nice practical application:

One of the main applications of paraunitary matrices is in signal processing where $\PU(b)^+$, $b:\mathbb{R}^n \times \mathbb{R}^n \to \mathbb{R}^n$ being the canonical inner product, is considered under the name of \emph{$n \times n$-FIR lossless matrices}. See \cite[Chapter 14]{Paraunitary2} as a reference.

Rump proved (\cite[Propositions 3,4]{rump_goml}) that in the structure monoid $S(X)$ of an arbitrary OML each element $x$ has a unique factorization as
\[
x = x_1 x_2 \ldots x_n
\]
with $x_i \in X$ such that $x_i^{\ast} \vee x_{i+1} = 1$ ($1 \leq i < n$). This means it is not possible to right-contract the factorization by factoring any $x_i = y_1y_2$ ($y_1,y_2 \in X$) such that $y_2x_{i+1}\in X$, therefore enlarging $x_{i+1}$.

Existence and uniqueness for a factorization that is not left-contractable is shown in a more general context in \cite[Proposition 12]{Rump-Geometric-Garside}.

Using \autoref{cor:semidirect_product} we therefore conclude that for any $p \in \PU(b)^-$ there exist unique subspaces $U_i$ ($1\leq i \leq n$) and a unique $h \in U(b)$ such that $U_i^{\ast} + U_{i+1} = V$ ($1 \leq i < n$) and
\[
p = p_{U_1} p_{U_2} \ldots p_{U_n} h.
\]
This factorization is given by factoring $p$ into a pure paraunitary and a paraunitary matrix and then applying Rump's normal form theorem to the paraunitary factor.

Clearly, there also exist normal forms with $h$ as leftmost factor and/or the opposite condition $U_i + U_{i+1}^{\ast} = V$.

For the classical paraunitary matrices - i.e. $PU(b)^+$, $b$ being the canonical inner product on $\mathbb{R}^n$ - this is a known result (\cite[14.4.2, 4]{Paraunitary2}).

\bibliographystyle{halpha}

\end{document}